\documentclass[12pt,a4paper]{amsart}
\usepackage{preamble}

\newcommand{\dd}[2]{\frac{d #1}{d #2}}

\newcommand{\ip}[1]{\left\langle #1 \right\rangle} 
\newcommand{\R}{\mathbb{R}}

\newcommand{\cV}{\mathcal{V}}
\newcommand{\cA}{\mathcal{A}}
\newcommand{\cM}{\mathcal{M}}

\newcommand{\cP}{\mathcal{P}}
\newcommand{\cL}{\mathcal{L}}
\newcommand{\cS}{\mathcal{S}}

\newcommand{\cX}{\mathfrak{X}}

\DeclareMathOperator{\grad}{grad}

\renewcommand{\div}{\diverg}

\renewcommand{\phi}{\varphi}
\newcommand{\om}{\omega}
\newcommand{\wcw}{\wedge\cdots\wedge}
\renewcommand{\epsilon}{\varepsilon}
\newcommand{\te}{{\tilde e}}

\newcommand{\eps}{\varepsilon}
\newcommand{\teps}{{\tilde \varepsilon}}

\newcommand{\tf}{{\tilde f}}
\newcommand{\tecu}{{\tilde e^{cu}}}

\newcommand{\tL}{{\tilde L}}

\newcommand{\slot}{{(\cdot)}}
\renewcommand{\div}{\textnormal{div}}

\title[Equivariant divergence formula for chaotic flows] 
{
Equivariant divergence formula for chaotic flows
}

\begin{document}

\begin{abstract}

We prove the equivariant divergence formula for the axiom A flow attractors, which is a recursive formula for perturbation of transfer operators of physical measures along center-unstable manifolds. Hence the linear response acquires an `ergodic theorem', which means that it can be sampled by recursively computing only $2u$ many vectors on one orbit, where $u$ is the unstable dimension.

\smallskip
\noindent \textbf{Keywords.}
linear response,
chaos, 
equivariant divergence,
adjoint shadowing.

\smallskip
\noindent \textbf{AMS subject classification numbers.}
37C40, 
47A05. 
37M25, 
90C31. 
\end{abstract}

\maketitle
\section{Introduction}

\subsection{Literature review}
\hfill
\vspace{0.1in}

The long-time statistic of a hyperbolic chaotic system is given by the SRB measure $\rho$, which is a model for fractal limiting invariant measures of chaos \cite{young2002srb,srbmap,srbflow}.
The linear response is the derivative of the SRB measure with respect to the parameters of the system.

There are several different formulas for the linear response.
The most well-known are the pathwise perturbation formula, the divergence formula, and the kernel differentiation formula \cite{Gallavotti1996,Jiang2012,Ruelle_diff_maps_erratum,Ruelle_diff_flow,Dolgopyat2004,Gouezel2006,Baladi2007,Hairer2010,Pollicott2016}.
The pathwise perturbation formula averages conventional adjoint formula over a lot of orbits, which is also known as the ensemble method or stochastic gradient method; however, it is cursed by the gradient explosion \cite{Lea2000,eyink2004ruelle,lucarini_linear_response_climate,lucarini_linear_response_climate2}. 
The divergence formula, also known as the transfer operator formula computes the perturbation operator, which is not pointwisely defined, and we have to partition the full phase space to obtain some mollified approximation: this is cursed by dimensionality \cite{Keane1998,Froyland2007,Liverani2001,Ding1994,Pollicott2000,Galatolo2014,Galatolo2014a,Wormell2019,Crimmins2020,Antown2022,Wormell2019a,Froyland2013a,SantosGutierrez2020,Bahsoun2018,Zhang2020,Pollicott2000}.
The kernel differentiation formula is also known as the likelihood ratio method in probability context \cite{Rubinstein1989,Reiman1989,Glynn1990,Hairer2010,np}; it works only for stochastic systems, and the cost is large when the noise is small.
In this paper, we consider only the deterministic flow with hyperbolicity.

The system dimension $M$ of most real-life dynamical systems is large; for example, a discretized 3-dimensional Navier-Stokes equation can easily have $M\ge10^6$.
The only way to overcome the curse of dimensionality is essentially the ergodic theorem: fix any $C^\infty$ observable function $\Phi:\cM \rightarrow \R$, then for almost all $x$ in the attractor basin $U$,
\[ \begin{split}
  \lim_{T\rightarrow \infty} \frac 1T \int_{t=0}^T \Phi (f^t x) 
  = \rho(\Phi)
  : = \int \Phi(x) \rho(dx).
\end{split} \]
This is both theoretically interesting and essential for high dimensional numeric since, roughly speaking, $\Phi$ reduces to a 1-dimensional random variable, and with some decay of correlations, we can ignore the details of the random variables when performing averages.
The trade-off here is that we have to confine our interest to the averaged observable, and we can not ask the details of $\rho$ at arbitrary locations.
By `ergodic theorem' type formulas, we refer to formulas that are averages of pointwise functions of a few vectors, which can be expressed by recursive relations on a orbit.
This paper provides such a formula for the linear response of hyperbolic flows with perhaps the least number of vectors being tracked.

A promising approach to obtain `ergodic theorem' type formula is blending the pathwise perturbation and divergence formulas, but the obstruction was that the split lacks smoothness.
It is natural to apply the pathwise perturbation formula to the stable or shadowing contribution of the linear response.
The formula tracking the least number of vectors is the nonintrusive shadowing formula, which involves only $u$ many vectors \cite{Ni_NILSS_JCP,Ni_nilsas,Ni_asl}.
The other such formulas are less efficient, and the earliest should be the stable part in the blended response formula, which evolves $M$ many vectors \cite{abramov2007blended}.

The unstable contribution is harder to sample on an orbit due to lack of regularities in directional derivatives.
When the ratio of unstable directions is low, with some additional statistical assumptions, the unstable contribution could sometimes be small, so we sometimes can still get a useful derivative even if we do not compute the unstable contribution \cite{Ruesha}.
Some pioneering works gave pointwisely defined formulas which are less well-known \cite{Ruelle_diff_maps_erratum,Gouezel2008}; but they still involve terms whose expressions are not obvious; moreover, they are not recursive; even if we use our tools to make those formulas recursive, the number of recursive relations would be large.
The equivariant divergence formula solved the regularity difficulty by evolving only $2u$ many vectors per step \cite{far,fr,TrsfOprt}.
But fast response does not yet have continuous-time versions.

The continuous-time cases are more important than discrete-time because most physical and engineering systems are continuous-time.
It is also more difficult due to the presence of the one-dimensional center direction parallel to the flow, where perturbations do not decay exponentially either forward or backward in time.
But we know the expression of the centre direction, and the stable and unstable are differentiable in the flow direction; these extra information allow us to assemble the shadowing, unstable, and the center contributions into formulas similar to discrete-time.

This paper does not fully solve the practical problem of computing an approximated and useful linear response for most high-dimensional continuous-time systems, but it should be a building block of a good solution.
In discrete time, it seems that the algorithm based on the equivariant divergence formula works whenever the existence of the linear response could be proved, including moderately nonuniform hyperbolic and moderately discontinuous case.
However, there are many counter examples where the linear response can be proved to not exist, and our algorithm explodes.
As for continuous time, most well-known systems, such as Lorenz systems, are not hyperbolic, and we guess that the linear response might not rigorously exist.
To remedy this, we have to introduce some approximation, and we proposed a trinity program in  \cite{Ni_asl}: for regions with bad hyperbolicity, we add local noise and switch to the kernel differentiation formula.

\subsection{Main results}
\hfill\vspace{0.1in}

This paper gives the equivariant divergence formula for the axiom A flow attractors, which is a recursive formula for perturbation of transfer operators of physical measures along center-unstable manifolds.
This further gives a recursive formula for the linear response, which has potential numerical applications.

For continuous-time, let $f$ be the flow of $F+\gamma X$, where both $F$ and $X$ are smooth, $F$ is the base flow, $X$ is the perturbation, and the parameter $\gamma$ has base value zero.
Let $X^{cu}$ and $X^s$ be center-unstable and stable part of $X$.
Let $\sigma$ be the conditional density function on of the unique physical SRB measure $\rho$ for a local center-unstable manifolds. Let $\tL^{cu}$ be the local transfer operator of $\xi\tf$, which is a map such that $\delta(\xi\tf)=X^{cu}$ (see \cref{s:cuL} for detailed definitions). 
So $\delta\tL^{cu}\sigma$ is the center-unstable perturbation of the transfer operator on $\sigma$. 
We take the ratio $-\frac{\delta\tL^{cu}\sigma}{\sigma}$, which will not depend on the choice of the neighborhood of  the local foliation. 
$-\frac{\delta\tL^{cu}\sigma}{\sigma}$ may as well be denoted by the unstable submanifold divergence, $\div^{cu}_\sigma X^{cu}$,
which has the equivariant divergence formula,

\begin{restatable}[equivariant divergence in continuous-time]{theorem}{goldbach}\label{t:eqdiv}
Let $\eta: = \eps^c(X)$, then
\[ \begin{split}
 - \frac{\delta \tL^{cu}\sigma} \sigma =  \div^{cu}_\sigma X^{cu} = \cS(\div^v \nabla F, \div^v F) X + \div^v X + F(\eta).
\end{split} \]
holds $\rho$ almost everywhere.
\end{restatable}

Here $\cS$ is the adjoint shadowing operator.
Roughly speaking it is the adjoint operator of the tangent shadowing operator $S$ under a special inner product which combine a (covector field, function) pair with a (vector field, function) pair: the detail is reviewed in \cref{s:asl}.
The tangent shadowing operator $S$ map a vector field to a (vector field, function) pair as below:
\[ \begin{split}
S: X \mapsto [v, \eta]
\quad \textnormal{where} \quad
\eta F=X^c, 
v=\int_0^\infty f_*^tX^s_{-t}dt-\int_{-\infty}^0f^t_*X^u_{-t}dt.
\end{split} \]
The equivariant divergence $\div^v$ is the contraction by the unstable vector and covectors; the functions $\div^v F,\div^v X$ is the equivariant divergence of the vector field $F,X$. The $\div^v \nabla F$ is a covector field defined similarly. 
All these notations will be explicitly defined later in \cref{s:UCnotations}.

\Cref{t:eqdiv} is an `ergodic theorem' type formula, which means that it can be calculated by the information on one typical orbit pointwise-recursively. 
The formula of this type is of great physical importance because in physical world we can only observe the phase space by one typical orbit. 
Also it is of application importance because the algorithm complexity of this type of formula can avoid the curse of dimensionality.

With \cref{t:eqdiv} , we can give a new formula for computing unstable contribution of linear response. 
Let $\rho$ be the SRB measure which coincide with the long-time statistic.
Fix a smooth objective function $\Phi$.
Denote $\delta:=\partial/\partial \gamma|_{\gamma=0}$, we decompose the linear response formula from \cite{Ruelle_diff_flow} into the shadowing and the unstable contribution,
\[ \begin{split}
  \delta \rho (\Phi)
  = \int_{0}^{\infty}\rho(X(\Phi_t))dt
  = SC - UC,
\end{split} \]
where $\Phi_t:=\Phi\circ f^t$, and
\[ \begin{split}
  SC :=\int_{0}^{\infty}\rho(X^{s}(\Phi_t))dt-\int_{-\infty}^{0}\rho(X^{cu}(\Phi_t))dt,
  \quad
  UC
  := \int_{-\infty}^\infty \rho\left(  X^{cu}(\Phi_t) \right) dt. 
\end{split} \]

\begin{restatable}[fast response formula for linear response]{proposition}{silverbach}\label{uc}
\[ \begin{split}
SC = \rho(X\cS(d\Phi,\Phi-\rho(\Phi))),
\quad
  UC = \lim_{W\rightarrow\infty} \rho(\phi \frac{\delta \tL^{cu}\sigma} \sigma),
  \quad\textnormal{where}\quad 
  \phi:=\int_{-W}^W \Phi\circ f^t dt .
\end{split} \]
\end{restatable}
This gives a new interpretation of the unstable part of the linear response by $\delta\tL^u$. 
Recall that the physical measure can be sampled by an typical orbit, while \cref{uc} shows how to compute the linear response recursively on an orbit. 
Hence, roughly speaking, \cref{uc} is the ergodic theorem for linear response.

Using \cref{t:eqdiv} and \cref{uc}, the linear response can be expressed by recursively computing only $O(u)$ many vectors on one typical orbit.
In fact, in the real computing process, we can sample the physical measure $\rho$ by an orbit. Then we can compute a basis of the unstable subspace by pushing forward $u$ many randomly initialized vectors while performing occasional renormalizations, on the same orbit we used to sample $\rho$.
Similarly, we can compute a basis of the adjoint unstable subspace.
With these two basis we can compute the equivariant divergence $\div^v$.
Also these two basis are the main data required by the nonintrusive shadowing algorithm for computing $\cS$ (see \cref{s:asl}).

This paper is organized as follows.
First, \cref{s:prelim} reviews the tangent linear response theory and the recent adjoint shadowing lemma.
Then \cref{s:secdiv} proves the equivariant divergence formula.
Finally, \cref{s:UC} uses the equivariant divergence formula to calculate the linear response.

\section{Preliminaries}
\label{s:prelim}

\subsection{Hyperbolicity and linear response}
\hfill\vspace{0.1in}

Let $f$ be the flow of $F+\gamma X$, where both $F$ and $X$ are smooth, $F$ is the base flow, $X$ is the perturbation, and the parameter $\gamma$ has base value zero.
Assume that 
\[ \begin{split}
  F\neq 0 \quad \textnormal{on } K.
\end{split} \]
Assume that the attractor $K$ is compact, $|F|$ is uniformly bounded away from zero and infinity.
We further assume $K$ is a mixing axiom A attractor. Now $T_KM$ has a continuous $f_*$-invariant splitting into stable, unstable, and center subspaces, $T_KM = V^s \bigoplus V^u \bigoplus V^c$, where $V^c$ is the one-dimensional subspace spanned by $F$, and 
\[ \begin{split}
  \max_{x\in K}|f_* ^{-t}|V^u(x)| ,
  |f_* ^{t}|V^s(x)| \le C\lambda ^{n}
  \quad \textnormal{for  } t\ge 0, \\
  \max_{x\in K}|f_* ^{t}|V^c(x)| \le C
  \quad \textnormal{for  } t\in \R.
\end{split} \]
Here $f_*$ is the pushforward operator on vectors.
Define the oblique projection operators $P^u$ and $P^s$, such that
\[ \begin{split}
  v = P^u v + P^s v, \quad \textnormal{where} \quad P^u v\in V^u, P^s v\in V^s.
\end{split} \]

For an orbit $x_t=f^t(x_0)$, the so-called homogeneous tangent equation of $e_t\in T_{x_t}\cM$ can be written by three equivalent notations:
\[ \begin{split}
  e_t  =f_*^t e_0 
  \quad\Leftrightarrow\quad \cL_F e = 0
  \quad\Leftrightarrow\quad \nabla_F e = \nabla_e F.
\end{split} \]
Here $\cL$ is the Lie-derivative and $\nabla$ is the Riemannian derivative.
The last expression is an ODE since $\nabla_F$ is typically denoted by $\partial/\partial t$ in $\R^M$.

A mixing axiom A attractor admits a physical SRB measure $\rho$.
Existence of linear responses and the formula was proved \cite{Ruelle_diff_flow,Dolgopyat2004}. Denote
\[
\delta(\cdot):=\partial (\cdot)/\partial \gamma |_{\gamma=0}
\quad
\textnormal{is the total derivative with all other parameters fixed.}
\]
Here `total derivative' means when $g=g(\gamma,x(\gamma,y))$, then $\delta g = \partial g/\partial \gamma + \partial g/\partial x \cdot \partial x /\partial \gamma|_{\gamma=0}$.
We sometimes write $\delta(\cdot):=\partial (\cdot)/\partial \gamma |_{\gamma=0}$ to emphasize that all other parameters are fixed, while sometimes write $\delta(\cdot):=d (\cdot)/d \gamma |_{\gamma=0}$ to emphasize that we expand all levels of dependence on $\gamma$.
The linear response can be given by a pathwise perturbation formula,
\begin{equation} \begin{split} \label{e:meiguo}
  \delta \rho (\Phi)
  = \lim_{T\rightarrow\infty} \rho\left( \int_{0}^T X(\Phi_t) dt \right) ,
\end{split} \end{equation}
where $X(\cdot)$ is to differentiate in the direction of $X$.

In this paper we shall use the above formula, the so-called `ensemble formula', as a starting point, and derive a linear response formula that can be sampled by $2u$ recursive relations on an orbit.
First, we decompose the linear response into
\[ \begin{split} \label{e:ruelle22}
  \delta \rho (\Phi)  = SC + UC, 
  \quad
  SC=\int_{0}^{\infty}\rho(X^{s}(\Phi_t))dt-\int_{-\infty}^{0}\rho(X^{cu}(\Phi_t))dt
\end{split} \]
which we call the shadowing and the unstable contribution.
When $u\ll M$, it might happen that $\delta\rho\approx SC$ \cite{Ruesha}.

\subsection{Adjoint shadowing lemma}
\label{s:asl}
\hfill\vspace{0.1in}

The adjoint shadowing lemma give three equivalent characterizations for the adjoint shadowing operator $\cS$ in continuous-time \cite{Ni_asl}.
Let $f^*$ be the pullback operator on covectors.
The adjoint flor $\{f^{*t}\}_{t\in\R}$ is also hyperbolic.
More specifically, we can show that the dual space of $V^s, V^c$, and $V^u$, denoted by $V^{*s}, V^{*c}$, and $ V^{*u}$, are also the stable, unstable, and center subspace. 
Define the image space of $\cP^u$ and $\cP^s$ as $V^{u*}$ and $V^{s*}$; we can show that they are in fact the stable, center, and unstable subspaces for the adjoint system.

Denote $(\cdot)_t:=(\cdot)(f^tx)$.
Let $\eps^c$ be the covector in the center subspace such that $\eps^c (F)=1$.
Let $\cL_{(\cdot)}(\cdot)$ be the Lie derivative, and $\nabla_{(\cdot)}(\cdot)$ be the Riemannian covariant derivative. 
Define the derivative of vector field $F$ along a covector $\nu$ as
\[
\nabla_\nu F := \nabla_F\nu-\cL_F\nu,
\]
 so $\nabla_\nu F = -(\nabla F)^T \nu$ when $\cM=\R^M$.

We first define the tangent shadowing operator. 
Denote the space of Holder vector and covector fields by $\cX^\alpha$ and $\cX^{*\alpha}$; define the quotient space
\[ \begin{split}
  A(K):=\{ (v,\eta): v\in \cX^\alpha (K), \nabla_F v\in  \cX^\alpha (K), \eta\in C^\alpha(K) \} \, / \sim.
\end{split} \]
where the equivalent relation $\sim$ is defined as
\[ \begin{split}
(v_1,\eta_1)\sim (v_2,\eta_2)
\quad \textnormal{iff} \quad 
  \cL_F v_1 +\eta_1 F = \cL_F v_2 +\eta_2 F.
\end{split} \]
We define the tangent shadowing operator $S:\cX^\alpha(K)\rightarrow A(K)$ as the 
\[ \begin{split}
S: X \mapsto [v, \eta],
\quad \textnormal{where} \quad 
\cL_F v + \eta F
= X .
\end{split} \]
Here $[v,\eta]$ is the equivalent class of $(v,\eta)$ according to $\sim$.
Roughly speaking, $v$ is the location difference and $\eta$ is the time rescaling between two shadowing orbits. 
A particular $v$ and $\eta$ can be given by the expression:
$$
\eta F=X^c, 
v=\int_0^\infty f_*^tX^s_{-t}dt-\int_{-\infty}^0f^t_*X^u_{-t}dt. $$

Then we define $\cA$, the dual space  of $A$.
Let $F(\psi)$ be the derivative of $\psi$ along $F$; define the space of pairs of a covector field and a scalar function, 
\[ \begin{split}
  \cA(K):=\{(\om, \psi) \,|\, 
  \om\in \cX^{*\alpha}, 
  \psi\in C^\alpha, 
  F(\psi) \in C^\alpha(K), 
  F(\psi) = \om(F)\}.
\end{split} \]
Then we can show that, for $(\om, \psi)\in\cA(K)$, $\llangle S(X);\om,\psi\rrangle$ is well-defined, where
\[ \begin{split}
  \llangle v,\eta;\om,\psi\rrangle 
  := \rho (\om v) - \rho(\eta \psi).
\end{split} \]
Hence we can well-define the shadowing contribution $SC$ as
\[ \begin{split}
  SC :=\llangle S(X); d\Phi, \Phi -\rho(\Phi) \rrangle,
\end{split}\]
All these definitions can all be made pathwise.

\begin{theorem}[adjoint shadowing lemma for continuous-time \cite{Ni_asl}]
\label{t:ASF}
On a compact mixing axiom A attractor with physical measure $\rho$ and exponential mixing,
the adjoint shadowing operator $\cS: \cA(K) \rightarrow \cX^{*}(K)$ is equivalently defined by the following characterizations:
  \begin{enumerate}
  \item 
  $\cS$ is the linear operator such that
  \[ \begin{split}
  \llangle S(X); \om, \psi \rrangle =\rho(X \cS(\om, \psi))
  \quad \textnormal{for any }
  X\in \cX^\alpha,
  \end{split} \]
  Hence, if $X=\delta f$,
  $\nu =\cS (d\Phi, \Phi -\rho(\Phi))$, 
  then the shadowing contribution is
\[ \begin{split}
  SC =\llangle S(X); d\Phi, \Phi -\rho(\Phi) \rrangle
  =\rho(X \cS(d\Phi, \Phi-\rho(\Phi))) 
  = \lim_{T\rightarrow \infty} 
  \frac 1{T} \int_{0}^T  \nu_t X_t dt \,.
\end{split} \]
  \item 
  $\cS$ has the `split-propagate' expansion formula
\begin{equation*}
  \cS (\om,\psi ) 
  = \int_{t\ge 0} f^{*t} \om^s_t dt
  -  \int_{t\le 0} f^{*t} \om^u_t dt
  - \psi \eps^c,
\end{equation*}
\item 
  The shadowing covector $\nu = \cS(\om,\psi)$ is the unique solution of the inhomogeneous adjoint ODE,
  \[ \begin{split}
  \nabla_F \nu - \nabla_\nu F 
  = \cL_F\nu
  = - \omega
  \; \textnormal{ on  } K,
  \quad \textnormal{} \quad 
  \nu F (x)  = - \psi (x)
  \; \textnormal{ at all or any } x\in K.
  \end{split} \]
  \end{enumerate}
Moreover, $\cS$ preserves Holder continuity.
\end{theorem}

Characterization (b) is the most powerful characterization for proving all properties of $\cS$.
In particular, the `split-propagate' scheme is a very common trick in hyperbolic systems to keep boundedness when pushing vectors or pulling covectors.
As we shall see, (b) appears in the unstable divergence, so we can apply the adjoint shadowing lemma.
On the other hands, characterization (a) shows that $SC$ can be expressed by $\cS$, and (c) is useful for numerically computing $\cS$ by the nonintrusive shadowing algorithm \cite{Ni_nilsas,Ni_NILSS_JCP,Ni_fdNILSS,Ni_CLV_cylinder}.

\section{Equivariant divergence formula in continuous-time}
\label{s:secdiv}

We derive a formula for the unstable contribution such that it can be sampled by $2u$ recursive relations on one orbit.

\subsection{Definitions and geometric notations}
\hfill\vspace{0.1in}
\label{s:UCnotations}

The unstable contribution is defined as the other part of the linear response, and in \cite{Ni_asl} we showed that
\begin{equation} \begin{split} \label{e:gaokao}
  UC := 
  \delta \rho(\Phi) - SC 
  = \int_{-\infty}^\infty \rho\left(  X^u(\Phi_t) \right) dt
  = \int_{-\infty}^\infty \rho\left(  X^{cu}(\Phi_t) \right) dt.
\end{split} \end{equation}
where $X^{cu}=X^c+X^u$. 
Without loss of generality, we can assume that $X^s \neq 0$ and $X^{cu} \neq 0$.
We use $\cV^{cu}(x)$ to denote the center-unstable manifold which contains $x$, and $\cV^{cu}_r(x)$ denotes the local center-unstable manifold of size $r$. 
The stable manifolds are $\cV^s(x)$, $\cV^s_r(x)$ and unstable are $\cV^u(x)$, $\cV^u_r(x)$.

Define the Hessian of a map, $\nabla f_*$, by the Leibniz rule, so for vector fields $X, Y, Z$, 
\[ \begin{split}
  (\nabla_{X} f_*) Y
  = \nabla_{f_* X } (f_*Y) - f_*\nabla_{ X} Y.
\end{split} \]
Denote the Hessian of vector fields by $\nabla^2$, so
\[ \begin{split}
  \nabla_X \nabla_Y Z
  = \nabla_{ \nabla_X Y} Z
  + \nabla^2_{X,Y} Z.
\end{split} \]
We can verify that the Hessians are symmetric,
\[ \begin{split}
  (\nabla_{X} f_*) Y 
  = (\nabla_{Y} f_*) X,
  \quad \textnormal{} \quad 
  \nabla^2_{X,Y} Z
  = \nabla^2_{Y,X} Z.
\end{split} \]

Denote a $u$-vector, which is basically a $u$-dimensional cube, by $ e:=e_1\wcw e_u $.
Recall that 
\[ \begin{split}
  \cL_Xe 
  := \sum_{i=1}^u e_1\wcw \cL_{X} e_i \wcw e_u.
\end{split} \]
Then we define some notations for derivatives, so that one slot in the derivative can take a cube, yet the rules still look like vectors.
\begin{equation} \begin{split}
\label{e:dengbing}
  \nabla_Xe 
  := \sum_{i=1}^u e_1\wcw \nabla_{X} e_i \wcw e_u ,
  \quad \textnormal{} \quad 
  \nabla_eX 
  := \sum_{i=1}^u e_1\wcw \nabla_{e_i}X \wcw e_u .
\end{split} \end{equation} 
So we still have $\cL_Xe = \nabla_Xe-\nabla_eX$.
Also, the Hessian of a map is
\[ \begin{split}
  (\nabla_{X} f_*) e 
  := \nabla_{f_* X } (f_*e) - f_*\nabla_{ X} e
  = \sum_{i=1}^u f_*e_1\wcw (\nabla_{X} f_*) e_i \wcw f_*e_u 
  \\
  = \sum_{i=1}^u f_*e_1\wcw (\nabla_{e_i} f_*) X \wcw f_* e_u
  =: (\nabla_{e} f_*) X .
\end{split} \]

Then we define Hessian of a vector field, where one slot takes a cube.
First recall that for 1-vector fields $X, Y$, and $Z$, we have
\[
\nabla^2_{X,Y} Z 
  := \nabla_X\nabla_Y Z - \nabla_{\nabla_XY}Z.
\]
Recursively apply \cref{e:dengbing}, also note that the wedge product is anti-symmetric, so interchanging two entries in a wedge product changes the sign, we have
\[ \begin{split}
  \nabla^2_{X,Y} e 
  := \nabla_X\nabla_Y e - \nabla_{\nabla_XY}e
  = \sum_{i\neq j}\cdots \nabla_Xe_i\cdots\nabla_Ye_j\cdots
  + \sum_i \cdots \nabla_X\nabla_Ye_i\cdots
  - \sum_i \cdots \nabla_{\nabla_XY}e_i\cdots 
  \\
  = \sum_i \cdots (\nabla_X\nabla_Ye_i -\nabla_{\nabla_XY}e_i) \cdots
  = \sum_i e_1\wcw \nabla^2_{X,Y}e_i \wcw e_u
  = \nabla^2_{Y,X} e .
\end{split} \]
\[ \begin{split}
  \nabla^2_{X,e} Y
  := \nabla_X\nabla_e Y - \nabla_{\nabla_Xe}Y
  = \sum_{i\neq j}\cdots \nabla_X e_j \cdots\nabla_{e_i} Y\cdots
  + \sum_i \cdots \nabla_X\nabla_{e_i}Y \cdots
  \\
  - \sum_{i\neq j}\cdots \nabla_Xe_i\cdots\nabla_{e_j}Y \cdots
  - \sum_i \cdots \nabla_{\nabla_X e_i} Y \cdots 
  = \sum_i e_1 \wcw \nabla^2_{X,e_i} Y \wcw e_u .
\end{split} \]
\[ \begin{split}
  \nabla^2_{e,X} Y
  := \nabla_e\nabla_X Y - \nabla_{\nabla_e X}Y
  = \sum_i \cdots \nabla{e_j} \nabla_{X} Y\cdots
  - \sum_i \nabla _{\cdots \nabla_{e_i} X \cdots} Y 
  \\
  = \sum_i \cdots \nabla{e_j} \nabla_{X} Y\cdots
  - \sum_{i\neq j} \cdots \nabla_{e_i} X \cdots \nabla_{e_j} Y \cdots
  - \sum_i \cdots \nabla _{\nabla_{e_i} X } Y \cdots
  \\
  = \sum_i \cdots \nabla{e_j} \nabla_{X} Y - \nabla _{\nabla_{e_i} X } Y\cdots
  = \sum_i \cdots \nabla^2_{e_i,X} Y \cdots
  = \nabla^2_{X, e} Y.
\end{split} \]
As expected, the Hessian of a map of a flow is related to the Hessian of the vector field that generates the flow.

\begin{lemma} [expression of $\nabla f_*$ by $\nabla^2 F$] \label{l:vari} Let $f^t$ be the flow generated by the vector field $F$, for any $x$, any $X\in T_x\cM$, any vector field $e$ defined at $x$ and differentiable along $X$,
\[ \begin{split}
  (\nabla_X f_*^t) e
  := \nabla_{f_*^t X} f_*^t e - f_*^t \nabla_Xe
  = \int_0^t f_*^{t-\tau} \nabla^2_{f_*^\tau X,f_*^\tau e} F_\tau d\tau.
\end{split} \]
\end{lemma}
\begin{proof}
Just check that both expressions solve the same ODE,
\[ \begin{split}
  (\cL_F v)_t = \nabla^2_{f_*^t X,f_*^t e} F_t,
  \quad \textnormal{} \quad 
  v_0 = 0. 
\end{split} \]
Let $v_t = \nabla_{f_*^t X} f_*^t e - f_*^t \nabla_Xe$ be the left side of the expression, then the Lie derivative of the second term along $F$ is zero, so
\[ \begin{split}
  \cL_{F_t} v_t
  = \cL_{F_t} \nabla_{f_*^t X} f_*^t e
  = \nabla_{F_t} \nabla_{f_*^t X} f_*^t e 
  - \nabla_{ \nabla_{f_*^t X} f_*^t e} F_t
  \\
  = \nabla^2_{F_t, f_*^t X} f_*^t e 
  + \nabla_{\nabla_{F_t} f_*^t X}  f_*^t e 
  - \nabla_{ \nabla_{f_*^t X} f_*^t e} F_t
  = \nabla^2_{F_t, f_*^t X} f_*^t e 
  + \nabla_{\nabla_{ f_*^t X} F_t}  f_*^t e 
  - \nabla_{ \nabla_{f_*^t X} f_*^t e} F_t
  \\
  = \nabla_{f_*^t X} \nabla_{F_t} f_*^t e 
  - \nabla_{ \nabla_{f_*^t X} f_*^t e} F_t
  = \nabla_{f_*^t X} \nabla_{f_*^t e} F_t 
  - \nabla_{ \nabla_{f_*^t X} f_*^t e} F_t
  = \nabla^2_{f_*^t X, f_*^t e} F_t 
\end{split} \]
Then let $v_t $ be the right side, $\int_0^t f_*^{t-s} \nabla^2_{f_*^s X,f_*^s e} F_s ds$, to verify the ODE.
\end{proof} 

The rules above are independent of the choice of $e$.
Later in this paper, we mostly use $e$ to denote the unstable cube, unless otherwise noted.
We also define co-cube $\eps$, the unit cubes $\te$ and $\teps$, and cubes in the center-unstable subspace,
\[ \begin{split}
  e:=e_1\wcw e_u,\quad
  \te := e/|e|, \quad
  e^{cu} := e\wedge F, \quad
  \te^{cu} := e^{cu}/|e^{cu}|, 
  \quad \textnormal{ where }
  e^i\in V^{u};
\\
  \eps := \eps^1 \wcw \eps^u, \quad 
  \teps := \eps/\eps(\te), \quad
  \eps^{cu} := \eps\wedge \eps^c, \quad
  \teps^{cu} := \eps^{cu}/|\eps^{cu}|, 
  \quad \textnormal{ where }
  \eps^i\in V^{u*}.
\end{split} \]
Here $\eps^c\in V^{*c}$ and $\eps^c F = 1$, $|\cdot|$ is the volume induced by $\ip{\cdot, \cdot}$, which is the inner product between $u$-vectors.

Define the equivariant unstable divergence as
\[ \begin{split}
  \div^v X:= \teps \nabla_\te X,
\end{split} \]
so $\div^v X$ is a real number.
Define $\div^v \nabla F$, which is a covector, as
\[ \begin{split}
  \div^v \nabla F(X) := \teps \nabla^2_{\te, X} F.
\end{split} \]
Similarly, we can define the center-unstable equivariant divergence as 
\[ \begin{split}
    \div^{cv} F := \teps^{cu}\nabla_{\te^{cu}} F\\
  \div^{cv} \nabla F(Y) 
  := \teps^{cu}\nabla^2_{Y,\te^{cu}}F  
\end{split} \]
The derivative in this divergence hits only smooth quantities such as $X$ and $\nabla F$, so it is well-defined as a normal function.
The equivariant divergence is Holder continuous since $e$ and $\eps$ are Holder.

\subsection{Center-unstable transfer operator \texorpdfstring{$\tL^{cu}$}{Lcu}}
\hfill\vspace{0.1in}
\label{s:cuL}

We present the proof of theorem~\ref{t:eqdiv} by first deriving an equivariant divergence formula using $e^{cu}$.
This derivation is longer than that in the appendix, but here we only need to use textbook theorems.
Compared with the discrete-time case \cite{TrsfOprt}, the main difficulty here is to handle the extra center direction.

First note that if $x\in K$ then $\cV^{cu}(x)\subseteq K$. 
This can be deduced from considering the time-one map $f^1$ of flow $F$. 
$T_KM$ has a $f^1_*$-invariant splitting $V^s\bigoplus V^c\bigoplus V^u$ where $V^s$ is uniformly contracting, $V^u$ is uniformly expanding and $V^c$ is the flow direction. 
Particularly, $K$ is a partially hyperbolic attractor of $f^1$.
It is a textbook theorem that under such condition we have $\cV^u(x) \subseteq K$ if $x\in K$, where $\cV^u(x)$ is the unstable manifold of  both  $f^1$ and $F$. 
Further use $K$ is invariant under the flow, we can get $\cV^{cu}(x)\subseteq K$ if $x\in K$.

The perturbation $X^{cu}$ can be written as $X-X^s$.
If we change the measure on a center-unstable manifold $\cV^{cu}$ by the flow of $X^{cu}$ for a distance of $\gamma$, then this perturbation roughly equals to change the measure by $\xi\circ \tf$, where $\tf$ is the flow of $X$ for a distance of $\gamma$, and $\xi$ is the holonomy map projecting along the stable manifolds onto $\cV^{cu}$.
More precisely, $\delta(\xi\tf)=X^{cu}$, so the linear response caused by $X^{cu}$ is the same as $\delta(\xi\tf)$.
See \cref{f:alau} for pictorial explanations.
The following lemma, from \cite{TrsfOprt}, shows that the map $\xi\tf$ is locally well-defined.

\begin{lemma}[$\xi\tf$ is locally well defined  \cite{TrsfOprt}]\label{ty:2}
Given $r>0$ small enough, $\exists \gamma_0>0$ such that for any $|\gamma|<\gamma_0$ and any $x\in K$, the point $y=\xi\tf x$ uniquely exists, and $|x-y|<0.1r$. 
From now on, we always assume $|\gamma|<\gamma_0$. 
\end{lemma}

Now we define the conditional measure density $\sigma$ and the transfer operator $\tL^{cu}$. First we choose $r>0$ small enough such that lemma \ref{ty:2} holds and for each $x\in K$ its neighbourhood $B(x,r)\bigcap K$ has an local foliation by local unstable-center manifolds.
Since the physical SRB measure $\rho$ is absolutely continuous alone the unstable-center foliation, 
we can define $\sigma$ as the density of the conditional measure on a local unstable-center foliation; 
$\sigma$ is at least $C^1$ on a unstable leaf.

Define $\tL^{cu}$ as the local transfer operator of $\xi\tf:\cV^{cu} \rightarrow \cV^{cu}$; note that $\tilde\cdot$ indicates dependence on $\gamma$.
Let $r$ small enough as above; for each $x\in K$,
let $P:=(\xi\tf)^{-1}\cV^{cu}_{0.1r}(x)$,
we define $\tL^{cu}$ as the transfer operator from $C^0(P)$ density function space to $C^0(\mathcal{V}_{0.1r}^{cu}(x))$.
Since $\delta(\xi \tf) = X^{cu}$, $\delta \tL^{cu}$ is the perturbation by $X^{cu}$.
Let $y_{-\Delta T} = (\xi\tf)^{-1} x$ (the subscript $\Delta T$ will be defined later), then the pointwise definition of $\tL^{cu}$ on any density function $\sigma$ is
\begin{equation} \begin{split} \label{e:tlsig}
  \tL^{cu}\sigma (x)
  := \frac {\sigma} { |\xi_* \tf_{*}| } (y_{-\Delta T})
  = \frac {\sigma (y_{-\Delta T})} { |\tf_{*}(y_{-\Delta T})| \, |\xi_*(\tf y_{-\Delta T})| },
\end{split} \end{equation}
and we can interpret the last equality as dissecting $X^{cu}$ into $X$ and $-X^s$.
Equivalently, fix the factor measure and the foliation of $\cV^{cu}_{0.1r}(x)$, then $\tL^{cu}\sigma$ is the conditional density of the new measure obtained by $\xi_* \tf_{*}$-pushforward the old measure on $B(x,r)$.

Then we can define $\frac{\tL^{cu}\sigma}{\sigma} (x)$ and $\frac{\delta\tL^{cu}\sigma}{\sigma} (x)$.
Notice that the domain of $\sigma$ includes $\mathcal{V}_{r}^{cu}(x)$, and $P:=(\xi\tf)^{-1}\cV^{cu}_{0.1r}(x)\subset \cV^{cu}_r(x) \subset K\cap B(x,r)$, so we can define $\frac{\tL^{cu}\sigma}{\sigma}(x)$ on the smaller leaf $\cV^{cu}_{0.1r}(x)$ for the particular $r$.
Moreover, notice that both $\sigma|_{\cV_{0.1r}^{cu}(x)}$ and $\sigma|_P$, the source of $\tL^{cu} \sigma|_{\cV_{0.1r}^{cu}(x)}$, are restrictions of the same $\sigma$ from the same larger leaf $\cV_{r}^{cu}(x)$ of the same foliation.
Hence, we can expect that, in $\frac{\tL^{cu}\sigma}{\sigma} (x)$, the factor due to the selection of $B(x,r)$ would cancel.
Notice that $\sigma$, $\tL^{cu}\sigma$ and $\delta \tL^{cu}\sigma$ may differ by a constant factor, depending on the choice of the neighborhood of the  local foliation, so they can only be defined locally;
but the ratio $\frac{\tL^{cu}\sigma}{\sigma}$ and $ \frac{\delta \tL^{cu} \sigma}{\sigma}$ does not depend on that choice, since the constant factors are cancelled, so they are globally well-defined.

Fix the point of interest $x$ and denote $x_t:=f^tx$.
We define $y_t$ as the pseudo-orbit shadowing $x_t$, with a discontinuity $\tf$ at $t=0$ when $\gamma \neq 0$.
More specifically,
\[ \begin{split}
y_0:=\xi^{-1}x,
\quad
y_t:=f^t y_0 
\quad\textnormal{for}\quad
t\ge0,
\\
y_{-\Delta T}:=\tf^{-1}y_0,
\quad
y_{t}=f^{t+\Delta T}(y_{-\Delta T})
\quad\textnormal{for}\quad
t<0,
\end{split} \]
where $\Delta T$ is such that $f^{\Delta T}(y_{-\Delta T})\in \cV^{u}_{r}(x)$. 
By the invariance of stable and unstable manifolds, we have $y_t\in \cV^s(y_t)$ for $t\ge0$ and $y_t\in \cV^u(x_t)$ for $t\le 0$.
Note that $y_t$ depends on $(\gamma,t)$: 
when $t<0$, and $(\gamma,t)$ change sightly, $y_t$ form a small neighbourhood in a (at least $C^1$) two-dimensional submanifold of $\cV^{cu}(x)$; 
when $t>0$, $y_t$ form a small submanifold of $\cV^{cs}(x)$. 
Above definitions are illustrated in \cref{f:ty}.

\begin{figure}[ht] \centering
  \includegraphics[width=0.7\textwidth,center]{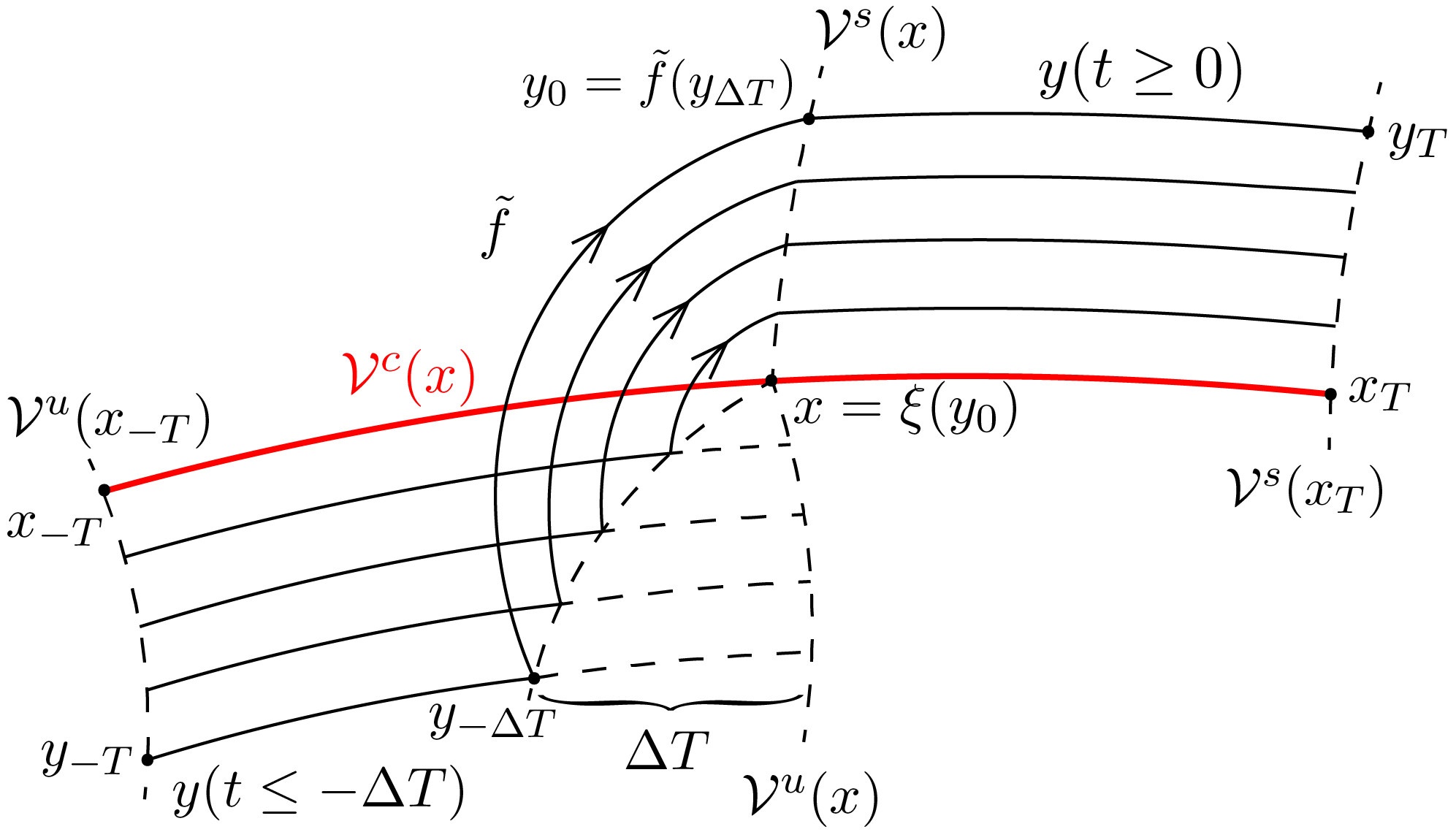}
  \caption{
  Pseudo-orbit $y_t$ for different $\gamma$. 
  }
  \label{f:ty}
\end{figure}

We define the vector field $e(y_{t})$ as
\[ \begin{split}
e(y_t):=
\begin{cases}
  f_*^{t+T} \te(y_{-T}), \quad\text{if } t<0;
  \\
  f_*^{t}\tf_*f_*^{T-\Delta T} \te(y_{-T}), \quad\text{if } t\ge0.
\end{cases}
\end{split} \]
When we use this notation in the following sections, $T$ is a fixed big number.
Similarly we can define $e^{cu}(y_t)$ as
\[ \begin{split}
e^{cu}(y_t):=
\begin{cases}
  f_*^{t+T} \te^{cu}(y_{-T}), \quad\text{if } t<0;
  \\
  f_*^{t}\tf_*f_*^{T-\Delta T} \te^{cu}(y_{-T}), \quad\text{if } t\ge0.
\end{cases}
\end{split} \]
When $\gamma=0$, we have $\tf=I_d,\tf_*=I_d$ and $\Delta T=0$. 
Denote $e^{cu}_t:= f_*^{t+T} \tilde e^{cu}(x_{-T})$ and $\eps^{cu}_t:=f^{*T+t} \teps^{cu}(x_{-T})$ for convenience, 
so $e^{cu}:=e^{cu}_0=e^{cu}(y_{-\Delta T})|_{\gamma=0}$ and $\eps^{cu}=\eps^{cu}_0$ .

\begin{figure}[ht] \centering
  \includegraphics[width=\textwidth,center]{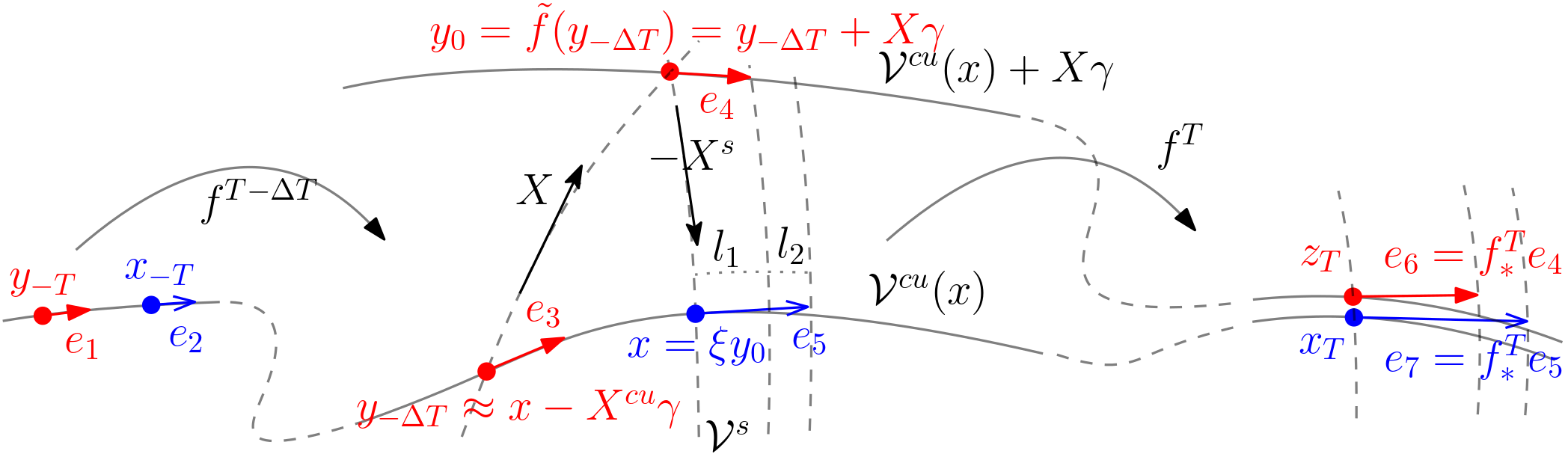}
  \caption{Measure transfer on $\cV^{cu}$.
  Here $y+X\gamma$ means to start from $y$ and flow along the direction of $X$ for a length of $\gamma$.
  We use $e$ as short of $e^{cu}$, and $|e_1| = |e_2|$.
  Roughly speaking, 
  $\delta \tL^u \sigma/\sigma (x) = l_2/l_1=(e_7-e_6)/e_7$.
  }
  \label{f:alau}
\end{figure}

Note that SRB is both invariant and also the pushforward of Lebesgue; it is nontrivial for the two definitions to coincide \cite{srbflow}.
If we accept both properties of SRB, then the conditional SRB measure on $\cV^{cu}(x)$, close to $x$, is obtained by pushing forward a $u+1$ dimensional Lebesgue measure on $\cV^u(f^{-T}x)$ for large $T$.
With this, we give a volume ratio formula for $\frac{\tL^{cu}\sigma}{\sigma}$, which solidifies the intuition in the caption of \cref{f:alau}.

\begin{lemma}[one volume ratio]\label{l:ty1}
Fix $x$ and the conditional density $\sigma$,  Let $\tL^{cu}$ be the transfer operator of $\xi\tf$, then we have the expression
\[ \begin{split}
  \frac{\tL^{cu}\sigma}{\sigma} (x) 
  = \lim_{T\rightarrow\infty} \lim_{T'\rightarrow\infty} 
  \frac{|f_*^{T'+2T}\te^{cu}(x_{-T})|}
  {|f_*^{T'}f_*^{T}\tf_*f_*^{T-\Delta T}\te^{cu}(y_{-T})|}.
\end{split} \]
\end{lemma}
\begin{proof}
First, recall that for SRB measures, for $y_{-\Delta T}\in \cV^{u}_r(x_{-\Delta T})$, the density $\sigma$ satisfies
\begin{equation} \begin{split} \label{e:srbsig}
\frac{\sigma(x_{-\Delta T})}{\sigma(y_{-\Delta T})} 
=
\lim_{t \rightarrow \infty} 
\frac
{\left|f_{*}^{t} \te^{cu}(f^{-t} y_{-\Delta T}) \right|}
{\left|f_{*}^{t} \te^{cu}(f^{-t} x_{-\Delta T}) \right|}.
\end{split} \end{equation}
This expression was stated and proved for example in \cite[proposition 1]{Ruelle_diff_maps_erratum} using unstable Jacobians.
Intuitively, this can be obtained by considering how the Lebesgue measure on $\cV^{cu}(x_{-t})$ is evolved.
The mass contained in the cube $\te^{cu}_{-t}$ is preserved via pushforwards, but the volume increased to $f_*^t \te^{cu}_{-t}$.
Hence, 
\[\begin{split}
\frac{\sigma(x_{-\Delta T})}{\sigma(y_{-\Delta T})} =\frac
{\left|f_{*}^{t} \te^{cu}(f^{-t} y_{-\Delta T}) \right|}
{\left|f_{*}^{t} \te^{cu}(f^{-t} x_{-\Delta T}) \right|}
\cdot \frac{\sigma(f^{-t}x_{-\Delta T})}{\sigma(f^{-t}y_{-\Delta T})}
\end{split}\]
Note that the conditional measure is determined up to a constant coefficient, which is canceled in the ratio.
Since $\lim_{t \rightarrow \infty} d(f^{-t}x_{-\Delta T},f^{-t}y_{-\Delta T})=0$, $\sigma$ is positive and continuous and $K$ is compact, so $\lim_{t \rightarrow \infty} \frac{\sigma(f^{-t}x_{-\Delta T})}{\sigma(f^{-t}y_{-\Delta T})}=1$ as $t\rightarrow\infty$, yielding \cref{e:srbsig}.

Then we consider the ratio $\frac{\sigma(x)}{\sigma(x_{-\Delta T})}$, where $x_{-\Delta T}:=f^{-\Delta T}x$.
Note that here the difference between $x$ and $x_{-t}$ is along the center direction.
In this case, due to the invariance of the SRB measure, we have
$$
\frac{\sigma(x)}{\sigma(x_{-\Delta T})}
=\frac{|f^{\Delta T}_*e^{cu}(x_{-\Delta T})|}{|e^{cu}(x_{-\Delta T})|}
=:|f^{\Delta T}_*(x_{-\Delta T})| ,
$$
for any $e^{cu}(x_{-\Delta T})$.
Combining with \cref{e:srbsig}, we get
\[
\frac{\sigma(y_{-\Delta T})}{\sigma(x)}
= \frac{\sigma(x_{-\Delta T})}{\sigma(x)}
  \cdot\frac{\sigma(y_{-\Delta T})}{\sigma(x_{-\Delta T})}
=  \lim_{T\rightarrow\infty}  
  \frac{|f_*^T\te^{cu}(x_{-T})|}
  {|f_*^{T-\Delta T}\te^{cu}(y_{-T})| }.
\]

Substituting into \cref{e:tlsig}, we get
\[ \begin{split}
  \frac{\tL^{cu}\sigma}{\sigma} (x) 
  =  \lim_{T\rightarrow\infty}  
  \frac{|f_*^T \te^{cu}(x_{-T})|}
  {|f_*^{T-\Delta T}\te^{cu}(y_{-T})| }
  \frac{1}{|\tf_{*}(y_{-\Delta T})| \, |\xi_*(\tf (y_{-\Delta T}))|}
  .
\end{split} \]
Here $|\tf_{*}(y_{-\Delta T})| := \frac{|\tf_* e^{cu}(y_{-\Delta T})|}{|e^{cu}(y_{-\Delta T})|}$, $|\xi_{*}(\tf(y_{-\Delta T}))| := \frac{|\xi_*\tf_* e^{cu}(y_{-\Delta T})|}{|\tf_*e^{cu}(y_{-\Delta T})|}$,
where $\tf_* e^{cu}(y_{-\Delta T})$ is a vector at $y_0=\tf y$, $\xi_*\tf_* e^{cu}(y_{-\Delta T})$ is a vector at $x$.

Then we give an expression of $|\xi_*|$.
By a corollary of the absolute continuity of the holonomy map \cite[theorem 4.4.1]{barreirapesin},
\begin{equation} \begin{split} \label{e:holo}
 \lim_{T'\rightarrow\infty}
\frac {|f^{T'}_* \xi_*\tf_* e^{cu}(y_{-\Delta T}) |}
{ |f_*^{T'} \tf_* e^{cu}(y_{-\Delta T})|}
= 1. 
\end{split} \end{equation}
Hence,
\begin{equation} \begin{split} 
\frac {1} {|\xi_*(\tf y_{-\Delta T})|} 
:= \frac {|\tf_* e^{cu}(y_{-\Delta T})|} {|\xi_*\tf_* e^{cu}(y_{-\Delta T})|} 
= \lim_{T'\rightarrow\infty}
\frac {|\tf_* e^{cu}(y_{-\Delta T})|} { |f_*^{T'} \tf_* e^{cu}(y_{-\Delta T})|}
\frac {|f^{T'}_* \xi_*\tf_* e^{cu}(y_{-\Delta T}) |} {|\xi_* \tf_* e^{cu}(y_{-\Delta T})|} .
\end{split} \end{equation}
By substitution and cancellation,
\[ \begin{split}
  \frac{\tL^{cu}\sigma}{\sigma} (x) 
  =  \lim_{T\rightarrow\infty}\lim_{T'\rightarrow\infty}
  \frac{|f_*^Te^{cu}(x_{-T})|}{|f_*^{T-\Delta T}e^{cu}(y_{-T})|  } 
  \frac{|e^{cu}(y_{-\Delta T})|}{|f_*^{T'} \tf_*e^{cu}(y_{-\Delta T})|} 
  \frac{|f_*^{T'}\xi_*\tf_* e^{cu}(y_{-\Delta T})|}{|\xi_*\tf_* e^{cu}(y_{-\Delta T})|} 
  \end{split}
  \]

To simplify the above expression, notice that both $f_*^T e^{cu}(x_{-T})$ and $\xi_*\tf_*e^{cu}(y_{-\Delta T})$ are in the one-dimensional subspace $\wedge^{cu} V^{cu}(x)$, so the growth rate of their volumes are the same when pushing forward by $f_*$, hence
\[ \begin{split}
  \frac {|f^{T'}_* \xi_*\tf_* e^{cu}(y_{-\Delta T}) |} {|\xi_* \tf_* e^{cu}(y_{-\Delta T})|}
  = \frac {|f^{T'+T}_* e^{cu}(x_{-T}) |} {| f_*^T e^{cu}(x_{-T})|}.
\end{split} \]
Similarly, 
\[ \begin{split}
  \frac{|e^{cu}(y_{-\Delta T})|}{|f_*^{T'} \tf_*e^{cu}(y_{-\Delta T})|} 
  = \frac{|f_*^{T-\Delta T} e^{cu}(y_{-T})|}{|f_*^{T'} \tf_* f_*^{T-\Delta T} e^{cu}(y_{-T})|}.
\end{split} \]
Finally, by substitution and cancellation, we have 
\[\begin{split}
  \frac{\tL^{cu}\sigma}{\sigma} (x) 
  = \lim_{T\rightarrow\infty} \lim_{T'\rightarrow\infty} 
  \frac{|f_*^{T'+T}\te^{cu}(x_{-T})|}
  {|f_*^{T'}\tf_*f_*^{T-\Delta T}\te^{cu}(y_{-T})|}.
\end{split}\]
For later convenience, we pass $T'$ to $T+T'$, yielding the expression in the statement.
\end{proof}

\subsection{Equivariant divergence by \texorpdfstring{$e^{cu}$}{e cu}}
\hfill\vspace{0.1in}
\label{s:divcu}

Then we take derivative of the volume-ratio formula.
There is a formula by integration by parts of~\eqref{e:gaokao}, which is 
\[ \begin{split}
  \delta L^{cu}\sigma = - \div^{cu} (\sigma X^{cu}),
\end{split} \]
where $\div^{cu}$ is the submanifold divergence on $\cV^{cu}$. 
The main issue is that $X^{cu}$ is not differentiable, so we need more work to find an `ergodic theorem' type formula, which can be sampled by progressively computing a few recursive relations on an orbit.

Roughly speaking, in figure~\ref{f:alau}, $ {\delta \tL^{cu}\sigma}/{\sigma} = (e_7-e_6)/e_7$, where $e_7-e_6=\nabla_{f_*^TX^s} e_T$.
Since $f_*^{T'}\tf_*f_*^{T-\Delta T}\te^{cu}(y_{-T})$ is almost in the unstable subspace after pushing forward for a long time, the difference $e_7-e_6$ is also in the unstable, so we can get the norm of this difference by just applying the unstable co-cube.
More specifically,

\begin{lemma} [one volume ratio for $\delta \tL$] \label{l:eps}
  \[ \begin{split}
  - \frac{\delta \tL^{cu}\sigma}{\sigma} (x) 
  = \lim_{T\rightarrow\infty} 
  \eps^{cu}_T \nabla_{f_*^TX^s} e^{cu}(y_T).
  \end{split} \]
\end{lemma}

\begin{proof}
Differentiate \cref{l:ty1} (this formal differentiation will be justified by its uniform convergence), and treat $e^{cu}(y_T) = f_*^{T}\tf_*f_*^{T-\Delta T}\te^{cu}(y_{-T})$ as a whole when applying the Leibniz rule, and notice that $e(y_t)|_{\gamma=0}=e_{t}$, we get
  \[ \begin{split}
  - \frac{\delta \tL^{cu}\sigma}{\sigma} (x) 
  = \lim_{T\rightarrow\infty} \lim_{T'\rightarrow\infty}-\delta\frac{|e^{cu}_{T+T'}|}{|e^{cu}(y_{T'+T})|}\\
  =\lim_{T\rightarrow\infty} \lim_{T'\rightarrow\infty}
  \frac{|e^{cu}_{T+T'}|\cdot(\delta|e^{cu}(y_{T+T'})|)}{|e^{cu}_{T'+T}|^2}\\
   =\lim_{T\rightarrow\infty} \lim_{T'\rightarrow\infty}
   \frac{|e^{cu}_{T+T'}|\cdot\delta \ip{e^{cu}(y_{T+T'}),e^{cu}(y_{T+T'})}^{\frac{1}{2}}}{|e^{cu}_{T'+T}|^2}
   \\
    =\lim_{T\rightarrow\infty} \lim_{T'\rightarrow\infty}
   \frac{|e^{cu}_{T+T'}|\cdot \frac{1}{2\ip{e^{cu}_{T+T'},e^{cu}_{T+T'}}} 2 \ip{\delta e^{cu}(y_{T+T'}),e^{cu}_{T+T'}}}{|e^{cu}_{T'+T}|^2}
   \\
  = \lim_{T\rightarrow\infty} \lim_{T'\rightarrow\infty}\frac{\ip{e^{cu}_{T+T'},\delta e^{cu}(y_{T'+T})}}{|e^{cu}_{T'+T}|^2}\\
  \end{split} \]
The equality holds because $e^{cu}_{T+T'}=e^{cu}(y{T+T'})|_{\gamma=0}$. For the forth equality to hold, here $\delta e^{cu}(y_{T+T'})=\nabla_{f_*^{T+T'}X^s}e^{cu}_{y_{T+T'}}$ should be the Riemannian derivative. Then 
    \[ \begin{split}
  - \frac{\delta \tL^{cu}\sigma}{\sigma} (x) 
  = \lim_{T\rightarrow\infty} \lim_{T'\rightarrow\infty}\frac{\ip{e^{cu}_{T+T'},\delta e^{cu}(y_{T'+T})}}{|e^{cu}_{T'+T}|^2}\\
   = \lim_{T\rightarrow\infty} \lim_{T'\rightarrow\infty}\frac{\ip{e^{cu}_{T+T'},\nabla_{f_*^{T+T'}X^s}e^{cu}_{T'+T}}}{|e^{cu}_{T'+T}|^2}\\
  = \lim_{T\rightarrow\infty} \lim_{T'\rightarrow\infty} 
  \frac{\ip{e^{cu}_{T'+T}, f_*^{T'}\nabla_{f_*^TX^s} e^{cu}_T + (\nabla_{f_*^TX^s} f_*^{T'}) e^{cu}_T }}{|e^{cu}_{T'+T}|^2}
  \end{split} \]

For the first term in the product, the so-called `relative decay' lemma in the appendix of~\cite{TrsfOprt}shows the uniform convergence of 
  \[ \begin{split}
  \lim_{T'\rightarrow\infty} 
  \frac{\ip{e^{cu}_{T'+T}, f_*^{T'}\nabla_{f_*^TX^s} e^{cu}_T}}{|e^{cu}_{T'+T}|^2}
  = \frac{1}{|e^{cu}_{T}|} \teps^{cu}_T\nabla_{f_*^TX^s} e^{cu}_T
  = \eps^{cu}_T \nabla_{f_*^TX^s} e^{cu}_T.
  \end{split} \]
Intuitively, this lemma says that only the unstable part of $\nabla_{f_*^TX^s} e^{cu}_T$ grows at the same speed as $e^{cu}_{T'+T}$, and all other parts are relatively negligible.

For the other term, first use lemma~\ref{l:vari}, then notice that the growth rate of any $u$-vector is bounded by the unstable $u$-vector by definition 
\[ \begin{split}
  \lim_{T\rightarrow\infty} \lim_{T'\rightarrow\infty}\frac{\ip{e^{cu}_{T'+T}, (\nabla_{f_*^TX^s} f_*^{T'}) e^{cu}_T }}{|e^{cu}_{T'+T}|^2}
  =\lim_{T\rightarrow\infty} \lim_{T'\rightarrow\infty} \ip{\te^{cu}_{T'+T}, \int_0^{T'} \frac 1 { |f_*^{T'-t}\te^{cu}_{T+t}| }f_*^{T'-t} \nabla^2_{f_*^{T+t} X^s, \te^{cu}_{T+t}} F_{T+t} dt }
  \\
  \le \lim_{T\rightarrow\infty} \lim_{T'\rightarrow\infty}\int_0^{T'} C |\nabla^2_{f_*^{T+t} X^s, \te^{cu}_{T+t}} F_{T+t}| dt 
  \le \lim_{T\rightarrow\infty} \lim_{T'\rightarrow\infty}\int_0^{T'} C \lambda^{t+T} dt
  \\
  \le \lim_{T\rightarrow\infty}\int_0^{\infty} C \lambda^{t+T} dt
  \le \lim_{T\rightarrow\infty} C \lambda^{T}
  = 0
\end{split} \]
where the values of the $C$'s above may change from expression to expression.
\end{proof} 

\begin{lemma} [expression of $\delta \tL$ by $\nabla f_*$] \label{l:star}
\[ \begin{split}
  - \frac{\delta \tL^{cu}\sigma}{\sigma} (x) 
  = \lim_{T\rightarrow\infty} 
  \eps^{cu} \nabla_{e^{cu}} X
  -\eta \eps^{cu}\nabla_{ e^{cu}} F 
  - \eps^{cu} ( \nabla_{f_*^{-T}X^{u}} f_*^T) e^{cu}_{-T} 
  + \eps^{cu}_T (\nabla_{X^s} f_*^{T}) e^{cu}
  .
\end{split} \]
\end{lemma}

\begin{proof}
Apply the Leibniz rule on $\nabla_{f_*^TX^s}e^{cu}(y_T) = \nabla_{f_*^TX^s}(f_*^{T}\tf_* e^{cu}(y_{-\Delta T}))$, and take value $\gamma=0 $ we get
\begin{equation} \begin{split} \label{e:da}
  - \frac{\delta \tL^{cu}\sigma}{\sigma} (x) 
  = \lim_{T\rightarrow\infty} 
  \eps^{cu}_T \left(
  (\delta f_*^{T}) e^{cu}
  + f_*^{T}(\delta \tf_*) e^{cu}
  + f_*^{T} \nabla_{-X^{cu}} e^{cu}(y_{-\Delta T})
  \right)
  \\
  = \lim_{T\rightarrow\infty} 
  \eps^{cu}_T (\delta f_*^{T}) e^{cu}
  + \eps^{cu} (\delta \tf_*) e^{cu}
  +   \eps^{cu} \nabla_{-X^{cu}} e^{cu}(y_{-\Delta T})
\end{split} \end{equation}
This equation can be intuitively explained in \cref{f:alau} as:
\[ \begin{split}
\gamma \nabla_{f_*^TX^s}(f_*^{T}\tf_* e^{cu}(y_{-\Delta T}))
\approx e_6-e_7
\\
= [(e_6-e_7)-f_*^T(e_4-e_5)]
 + [f_*^T(e_4-e_3)]
 + [f_*^T(e_3-e_5)],
\end{split}\]
if we multiply $\eps^{cu}$ on both sides, then the three terms on the right side correspond to the three terms in \cref{e:da}.

For the first term in \cref{e:da}, just use
\[ \begin{split}
  (\delta f_*^{T}) e^{cu}
  := (\nabla_{X^s} f_*^{T}) e^{cu}
\end{split} \]

For $(\delta \tf_*) e^{cu}$ in the second term in~\eqref{e:da},
\[\begin{split}
    (\delta \tf_*) e^{cu}
    :=\left(\left.
    \dd{}{\gamma}(\tf_*(\gamma,y_{-\Delta T})e^{cu}(y_{-\Delta T}))
    -\tf_*(\gamma,y_{-\Delta T})\dd{}{\gamma}e^{cu}(y_{-\Delta T})
    \right)\right|_{\gamma=0}
\end{split}\]
The first term equals the sum of partial derivatives with respect to $\gamma$ and $y_{-\Delta T}$, where  $y_{-\Delta T}$ further depends on $\gamma$.
Notice that $\tf_*(\gamma=0,y_{-\Delta T})=Id$ and $(y_{-\Delta T})|_{\gamma=0}=x$, so
\[\begin{split}
    (\delta \tf_*) e^{cu}
    :=\left(\left.
    \dd{}{\gamma}(\tf_*(\gamma,x)e^{cu}(x))
    +\dd{}{\gamma}(\tf_*(0,y_{-\Delta T})e^{cu}(y_{-\Delta T}))
    -\tf_*(\gamma,y_{-\Delta T})\dd{}{\gamma}e^{cu}(y_{-\Delta T})
    \right)\right|_{\gamma=0},
\end{split}\]
and the last two terms cancel.
Since $\tf$ is the flow of $X$, we can use the Lie bracket statement $\cL_X (\tf_* e^{cu})=0$, to get 
\[\begin{split}
    (\delta \tf_*) e^{cu}
    =\left(\left.
    \dd{}{\gamma}(\tf_*(\gamma,x)e^{cu}(x))
    \right)\right|_{\gamma=0}
    =\nabla_{X}(\tf_*e^{cu})
    =\nabla_{e^{cu}} X .
\end{split}\]

For the last term in \cref{e:da}, $\nabla_{X^{cu}} e^{cu}(y_{-\Delta T}) = \nabla_{X^{c}} e^{cu} + \nabla_{X^{u}} e^{cu}$.
Then, $ \nabla_{X^{c}} e^{cu} = \eta \nabla_F e^{cu} = \eta \nabla_{ e^{cu}} F$, where $\eta:=\teps^c X$. This is because by definition, $e^{cu}_{t}$ is generated by the push-forward of the flow $F$ when $ t\leq 0$, so we have $\cL_Fe^{cu}=\nabla_Fe^{cu}-\nabla_{e^{cu}}F=0$.
Also,
$\nabla_{X^{u}} e^{cu} = f_*^T \nabla_{f_*^{-T}X^{u}} e^{cu}_{-T} + ( \nabla_{f_*^{-T}X^{u}} f_*^T)  e^{cu}_{-T}$ by the definition of  $\nabla_{f_*^{-T}X^{u}} f_*^T$. 
Since $f_*^{-T}X^{u} \rightarrow 0$,
\[ \begin{split}
  \lim_{T\rightarrow\infty} \eps^{cu} f_*^{T}\nabla_{f_*^{-T}X^{u}} e^{cu}_{-T}
  = \lim_{T\rightarrow\infty} \eps^{cu}_{-T} \nabla_{f_*^{-T}X^{u}} e^{cu}_{-T}
  =0,
\end{split} \]
Hence,
\[ \begin{split}
  \lim_{T\rightarrow\infty} 
  \eps^{cu} \nabla_{X^{cu}} e^{cu}  
  = \lim_{T\rightarrow\infty} 
  \eps^{cu} \left(
  \eta \nabla_{ e^{cu}} F 
  + ( \nabla_{f_*^{-T}X^{u}} f_*^T) e^{cu}_{-T}.
  \right)
\end{split} \]
To summarize, we transformed~\eqref{e:da} to the expression in the lemma.
\end{proof} 

\begin{theorem}
[center-unstable equivariant divergence formula] \label{l:paiqiu}
\[ \begin{split} 
  - \frac{\delta \tL^{cu}\sigma}{\sigma} (x) 
  = \div^{cv} X
  + \cS (\div^{cv} \nabla F, \div^{cv} F) X
  .
\end{split} \]
\end{theorem}

\begin{proof}
Applying lemma~\ref{l:vari}, $\nabla f_*^T$ in the last two terms of lemma~\ref{l:star} become
\[ \begin{split}
  - \frac{\delta \tL^{cu}\sigma}{\sigma} (x) 
  = \lim_{T\rightarrow\infty} 
   (\eps^{cu} \nabla_{e^{cu}} X
  -\eta \eps^{cu}\nabla_{ e^{cu}} F  \\ 
  - \eps^{cu} \int_0^T f_*^{T-t} \nabla^2_{f_*^{t-T} X^{u}, e^{cu}_{t-T}} F_{t-T} dt
  + \eps^{cu}_T \int_0^T  f_*^{T-t} \nabla^2_{f_*^t X^s, e^{cu}_t } F_t dt )
  \\
  = \teps^{cu} \nabla_{\te^{cu}} X
  -\eta \teps^{cu}\nabla_{\te^{cu}} F 
  - \int_{-\infty}^0 \teps^{cu}_{t} \nabla^2_{f_*^{t} X^{u}, \te^{cu}_{t}} F_{t} dt
  + \int_0^\infty \teps^{cu}_{t} \nabla^2_{f_*^t X^s, \te^{cu}_t } F_t dt 
  .
\end{split} \]
Denote covector field $\om$ and function $\psi$ on $K$ such that 
\[ \begin{split}
  \om(Y):=\div^{cv} \nabla F(Y) 
  := \teps^{cu}\nabla^2_{Y,\te^{cu}}F,
  \quad \textnormal{} \quad 
  \psi := \div^{cv} F := \teps^{cu}\nabla_{\te^{cu}} F.
\end{split} \]
Then we can get the `split-propagate' scheme in characterization (b) of the adjoint shadowing operator in theorem~\ref{t:ASF},
\[ \begin{split} 
  - \frac{\delta \tL^{cu}\sigma}{\sigma} (x) 
  = \div^{cv} X
  + \left(
  - \psi \eps^c 
  - \int_{-\infty}^0  f_*^{t}\om^u_{t} dt
  + \int_0^\infty  f_*^t \om^s_{t}    dt 
  \right) X
  .
\end{split} \]

To check $(\om,\psi) \in \cA$, once again notice that $\nabla_Fe^{cu} = \nabla_{e^{cu}} F$, and $\eps^{cu}_t$ is invariant under the push-forward of the flow $F$, we have
\begin{equation} \begin{split} \label{e:ucsb}
  F(\psi) = F(\teps^{cu}\nabla_{\te^{cu}} F) 
  = \eps^{cu} \cL_F(\nabla_{e^{cu}} F) 
  = \eps^{cu} (\nabla_F(\nabla_{e^{cu}} F) - \nabla_{\nabla_{e^{cu}} F}F) 
  \\
  = \eps^{cu} (\nabla^2_{F,e^{cu}}F + \nabla_{\nabla_Fe^{cu}}F - \nabla_{\nabla_{e^{cu}} F}F) 
  = \eps^{cu} \nabla^2_{F,e^{cu}}F 
  = \om(F).
\end{split} \end{equation}
Hence, the adjoint shadowing lemma applies and the theorem is proved.
\end{proof}

\subsection{Removing center direction from \texorpdfstring{$e^{cu}$}{e cu}}
\hfill\vspace{0.1in}
\label{s:divu}

This subsection seeks to further simplify the equivariant divergence formula by using the unstable cube $e$ instead of the center-unstable cube $e^{cu}$.
This formula might be beneficial for potential numerical applications.

\goldbach*

\begin{remark*}
(1)
This formula is given by $u$ many recursive relations on an orbit, which should be optimal.
(2)
$\rho(F(\eta))=0$. 
To see this, choose an $x\in K$ such that $\lim\limits_{T\to\infty} \frac{1}{T}\int_0^T \Phi(f^t(x)) dt= \rho(\Phi)$ for every continuous function $\Phi$.
Take $\Phi=F(\eta)$ and we have 
\[ \begin{split}
  \rho(F(\eta)) 
  = \lim_{T\rightarrow\infty}\frac1T\int_0^TF(\eta)(f^t(x)) dt
  = \lim_{T\rightarrow\infty}\frac1T(\eta(f^Tx)-\eta(x))
  = 0.
\end{split} \]
\end{remark*}

\begin{proof}
In this proof we let $\te^{cu}:=\te\wedge F$, $\teps^{cu}:=\teps\wedge\eps^c$, where $\eps^c \in V^{*c}$ and $\eps^c(F)=1$.
Now $|\te^{cu}|\neq1$ but still $\teps^{cu}(\tecu)=1$. 
Notice that $\te^{cu}$ and $\teps^{cu}$ always appear in pairs and the constant factors cancel, so the equivariant divergence formula still holds with the new $\teps^{cu}$ and $\te^{cu}$.
Since $F\perp V^{*u}$, $\eps^c\perp V^u$, 
\[ \begin{split}
  \div^{cv} X = \teps\wedge\eps^c(\nabla_{\te\wedge F}X)
  = \teps \wedge \eps^c(  \te\wedge\nabla_FX )
  + \sum_{i=1}^u \teps \wedge \eps^c (e_1\wcw\nabla_{e_i} X \wcw e_u\wedge F)
  \\
  = \eps^c\nabla_FX + \teps \nabla_{\te} X 
  = \eps^c\nabla_FX + \div^v X .
\end{split} \]
Similarly,
\[ \begin{split}
  \div^{cv} F 
  = \eps^c\nabla_FF + \div^v F.
\end{split} \]
Also, due to invariance of $V^{cu}$ and $V^s$,
\[ \begin{split}
  (\div^{cv} {f^T_*})(\cdot)
  := \teps^{cu} (\nabla_{\slot} f_*^T) \te^{cu}_{-T}
  = \teps (\nabla_{\slot} f_*^T) \te_{-T}
  + \eps^c (\nabla_{\slot} f_*^T) F_{-T} .
\end{split} \]

Substituting into lemma~\ref{l:star}, we get
\begin{equation} \begin{split} \label{e:huasheng}
  - \frac{\delta \tL^{cu}\sigma}{\sigma} (x) 
  = \lim_{T\rightarrow\infty} 
  (\div^v X 
  -\eta \div^v F 
  - \eps ( \nabla_{f_*^{-T}X^{u}} f_*^T) e_{-T} 
  + \eps_T (\nabla_{X^s} f_*^{T}) e
  \\
  + \eps^c\nabla_FX
  - \eta \eps^c\nabla_F F
  - \eps^c ( \nabla_{f_*^{-T}X^{u}} f_*^T) F_{-T} 
  + \eps^c_T (\nabla_{X^s} f_*^{T}) F
  .
\end{split} \end{equation}
By the same proof of proposition~\ref{l:paiqiu}, the first line in~\eqref{e:huasheng} can be expressed by an equivariant divergence in the unstable manifold.

For the terms in the second line of~\eqref{e:huasheng}, when $T\rightarrow\infty$
\[ \begin{split}
  \eps^c ( \nabla_{f_*^{-T}X^{u}} f_*^T) F_{-T} 
  = \eps^c \nabla_{X^{u}} F -  \eps^c_{-T} \nabla_{f_*^{-T}X^{u}} F_{-T}
  \rightarrow \eps^c \nabla_{X^{u}} F ;
  \\
  \eps^c_T (\nabla_{X^s} f_*^{T}) F
  =  \eps^c_T \nabla_{f_*^{T}X^s} F_T - \eps^c\nabla_{X^s} F
  \rightarrow - \eps^c\nabla_{X^s} F.
\end{split} \]
Also,
\[ \begin{split}
  \eta \eps^c\nabla_F F
  = \eps^c\nabla_{\eta F} F
  = \eps^c\nabla_{X^c} F.
\end{split} \]
Hence, the second line equals
\[ \begin{split}
  \eps^c\nabla_FX
  - \eta \eps^c\nabla_F F
  - \eps^c \nabla_{X^{u}} F 
  - \eps^c\nabla_{X^s} F
  = \eps^c(\nabla_FX- \nabla_{X} F)
  = \eps^c\cL_FX
  = F(\eta),
\end{split} \]
where the last equality is because $F(\eta)=\cL_F(\eps^c X)$ and $ \cL_F \eps^c =0$.
This proves the expression in the lemma.
\end{proof} 

\hfill\vspace{0.1in}

\section{Unstable contribution of linear response}
\label{s:UC}

This section proves \cref{uc}, which shows that we can use the equivariant divergence to express the unstable contribution of linear response.

\silverbach*

\begin{proof}
The first equation is due to \cref{t:ASF}.
To prove the second equation, we only need to prove that for any $t$,
$$
\rho(X^{cu}(\Phi_t))=\rho(\Phi\circ f^t\frac{\delta\tL^{cu}\sigma}{\sigma}),
$$
and take integral $\int_{-\infty}^{\infty} dt$ of both sides of the equation.
We shall prove this by integration by parts on local center-unstable foliation in sets from a partition $R$ of $K$, whose construction is given below.

First define some notations for the local product structure.
We use $B(x,r)$ and $\Bar{B}(x,r)$ to denote the open ball and the closed ball center at $x$ of radium $r$.
For each $x\in K$, we choose a $r'_x>0$ so small that the neighbourhood $B(x,r'_x)\bigcap K$ has a local center-unstable foliation. 
There is a local product structure in this foliation. 
For each $a\in \cV^{cu}_{r}(x)$ and $b\in \cV^{s}_r(x)\bigcap K$ we denote $[a,b]$ as the unique transverse intersection of $\cV^{s}_{loc}(a)$ and $\cV^{cu}_{loc}(b)$ in $B(x,r'_x)\bigcap K$: this is well-defined when $r$ is small enough. 
Conversely, for each $x'$ near $x$, by the continuity of the stable and center-unstable manifolds, there exist unique 
$a=\cV^{cu}_r(x) \bigcap \cV^{s}_{loc}(x')=: \pi^{cu}_x(x')$ and
$b=\cV^{s}_r(x) \bigcap \cV^{cu}_{loc}(x') =: \pi^{s}_x(x')$
such that $[a,b]=x'$, and $a$ and $b$ both depend on $x'$ continuously: they are sometimes called the local coordinate of $x'$.

Then, for each point $x$, we define rectangles containing $x$.
For two sets $A\subseteq \cV^{cu}_{r}(x)$ and $B\subseteq \cV^{s}_r(x)\bigcap K$, we call the set
$[A,B]:=\{[a,b]: a\in A,b\in B\}$ a rectangle. 
Choose an $r_x<r'_x$ small enough.
We take $A=B(x,r_x)\bigcap\cV^{cu}_{loc}(x)$ and $B=B(x,r_x)\bigcap\cV^{s}_{loc}(x)\bigcap K$ which are open in $\cV^{cu}_{loc}(x)$  and $\cV^{s}_{loc}(x)\bigcap K$, $\Bar{A}=\Bar{B}(x,r_x)\bigcap\cV^{cu}_{loc}(x)$ and $\Bar{B}=\Bar{B}(x,r_x)\bigcap\cV^{s}_{loc}(x)\bigcap K$ which are closed in $\cV^{cu}_{loc}(x)$  and $\cV^{s}_{loc}(x)\bigcap K$ . 
Then $r_x$ can be so small that $[A,B]$ and $[\Bar{A},\Bar{B}]$ are well-defined rectangles in $B(x,r'_x)$.

We claim that $[A,B]$ is open in $K$ and  $[\Bar{A},\Bar{B}]$ is closed in $K$.
To see that, first notice that if $y\in [A,B]$ then $\pi^{cu}_x(y) \in A$. 
Then for any $y'\in K$ close  enough to $y$, $\pi^{cu}_x(y')\in A$, since $A$ is open and $\pi^{cu}_x$ is continuous. Also, for any $y'\in K$ close to $y$, $\pi^{s}_x(y') \in B$. 
Hence $y'\in [A,B]$, which imply $[A,B]$ is open in $K$. 
Similarly, $[\Bar{A},\Bar{B}]$ is closed in $K$.

Then we construct the partition $R$ from an open cover $R'$, which consists of rectangles with negligible boundaries.
Denote $R'(x)=[A,B]$ and $cl(R'(x))=[\Bar{A},\Bar{B}]$. 
Define $\partial R'(x)=cl(R'(x))-R'(x)$.
Then $r_x$ can be chosen such that $\partial R'(x)$ has zero physical SRB measure. 
This is because by definition, $\partial R'(x)$ will not intersect for different $r_x$; and $r_x$ has uncountable options while the physical measure is finite.
Then $\{{R'(x):x\in K}\}$ forms a open cover of $K$.
$K$ is compact so we can choose a finite subcover ${R'_1,R'_2,\dots,R'_m}$. 
We define $R_i= (R'_i-\bigcup_{k=1}^{i-1} cl(R'_k)),1\leq i\leq m$, then $R_i$ are pairwise disjoint and $\rho(K-\bigcup_{i=1}^{m}R_i)=0$. 
Define the partition $R:=\{R_i \}_{1\leq i\leq m}$, and this is the partition we want.

Then we can do integration by parts.
Let $\sigma'$ denote the factor measure and $\sigma$ denote the conditional density of $\rho$ with respect to this foliation. 
Notice that $\partial R := K-\bigcup_{i=1}^m R_m$ has zero physical measure, we take the following integration on $\bigcup_{i=1}^mR_i$. 
\[\begin{split}
    \rho(X^{cu}(\Phi_t))=\iint \sigma X^{cu}\cdot\grad(\Phi\circ f^t)dxd\sigma'(x).
\end{split}\]
Integrate-by-parts on unstable manifolds, note that the flux terms on the boundary of the partition cancel, so 
\[\begin{split}
    \rho(X^{cu}(\Phi_t))
    =\iint -\div_\sigma^{cu}X^{cu}(\Phi\circ f^t)\sigma dxd\sigma'(x)
    =\rho(\Phi\circ f^t\frac{\delta\tL^{cu}\sigma}{\sigma}).
\end{split}\]
Take it back to the formula and we can finish the proof.
\end{proof}

\appendix

\section{Derivation only in the unstable}
\label{a:u}

Our proof in the main body of the paper feels like taking a detour, since we started on the center-unstable manifold and then removed the center direction.
In this appendix we do not take that detour, and prove a similar formula by working only in the unstable manifold.
This proof is shorter and intuitive, but is only formal.
Nevertheless, it is consolidating to see that two proofs give equivalent formulas.

We first mention the main difficulties that prevent making the proof in the appendix rigorous. The first difficulty is the lack of smoothness. 
For example, notice that $\nabla_e X^c$ is not well-defined in general, since the oblique projection to $V^c$ also depends on $V^{su}$, which is not differentiable in general cases.
However, here we only do the formal calculate, and we assume that every object we use is smooth enough and all the notions are well-defined. 
Another main difficulty is the `absolutely continuous property of the holonomy map' we want to use here is non-standard: we shall explain this in the formal proof.

We add some new notations similar to those in the main body of the paper, but in unstable version instead of center-unstable version. 
We fix the point of interest $x$, choose a local unstable foliation, denote the density of the conditional SRB measure on  $\cV^u_r(x)$ at $\gamma=0$ by $\sigma^u$.
Similarly, we can define the local unstable transfer operator $\tL^u$, which is the transfer operator of $\tf\xi': \cV^u(x)\to\cV^u(x)$, $\tf$ is still the time one map of the flow $\gamma X$ and $\xi'$ is the holonomy map of center-stable foliation. 
Then we can define the ratio $\frac{\tL^u\sigma^u}{\sigma^u}$ and $\frac{\delta \tL^u\sigma^u}{\sigma^u}$. 
We will calculate the latter one.

\begin{theorem}
    \label{t:divuu}
\[ \begin{split}
  \frac{\delta \tL^u\sigma^u} {\sigma^u}
  = -\div^v X - \cS(\div^v\nabla F,\div^vF)X.
\end{split} \]
\end{theorem}

\begin{proof}[An indirect proof for $\rho$ almost everywhere]
We give a proof of the theorem, but only for $\rho$ almost everywhere, based based on  on \cref{t:eqdiv}; it also shows that the two theorems do not contradict.
First notice that, in a local center-unstable foliation, each center-unstable manifold $\cV^{cu}_{loc}$ is foliated by the unstable foliation $\cV^u_{loc}$, and the map $\tf\xi':\cV^{cu}\to\cV^{cu}$ preserve the unstable foliation $\cV^u$. 
So $\tL^u$ can both be understood as a local transfer operator of the map $\tf\xi':\cV^u(x)\to \cV^u(x)$ or $\tf\xi':\cV^{cu}(x)\to\cV^{cu}(x)$, and an interesting fact is that for $\rho$ almost everywhere,
\begin{equation} \label{e:comp}
\frac{\tL^u\sigma^u}{\sigma^u} 
= \frac{\tL^u\sigma^{cu}}{\sigma^{cu}},
\quad
\frac{\delta \tL^u\sigma^u}{\sigma^u}
= \frac{\delta \tL^u\sigma^{cu}}{\sigma^{cu}},
\end{equation}
where $\sigma^{cu}$ is the conditional density in center-unstable manifold (the main body of the paper uses $\sigma$ to denote the same quantity).
Intuitively, this is because in our framework, the measure change is entirely accounted for by the change of the conditional density, so they all must equal.
By \cref{e:comp}, we have
\[ \begin{split}
\frac{\delta \tL^u\sigma^u} {\sigma^u}
= \frac{\delta \tL^u\sigma^{cu}} {\sigma^{cu}}
= \frac{\delta \tL^{cu}\sigma^{cu}} {\sigma^{cu}}
-\frac{\delta \tL^{c}\sigma^{cu}} {\sigma^{cu}}.
\end{split} \]

By substitution of \cref{t:eqdiv} into the above equation, we see that now we only need to prove
\begin{equation}
\label{e:wew}
     \frac{\delta \tL^{c}\sigma^{cu}}{\sigma^{cu}} (x) 
 =- F(\eta)
 \quad 
 \rho
 \textnormal{ almost everywhere.}
\end{equation}
To see this, pick any smooth observable $\Phi$, choose an regular point $x_0$, such that for any continuous function $g$, $\lim_{T\rightarrow\infty}\frac{1}{T} \int_0^T g(f^tx_0) dt=\rho(g)$, and we have
\[ \begin{split}
  \rho(\Phi F(\eta)) + \rho (X^c (\Phi))
  = \rho(\Phi F(\eta) + \eta F (\Phi))
  = \rho( F(\Phi\eta))
  \\
  = \lim_{T\rightarrow\infty}\frac 1T \int_0^T F(\Phi\eta)(f^tx_0) dt
  = \lim_{T\rightarrow\infty}\frac 1T ((\Phi\eta)(f^{T}x_0) - (\Phi\eta)(x_0))
  = 0.
\end{split} \]
Hence
\[ \begin{split}
  - \rho(\Phi F(\eta))
  = \rho (X^c (\Phi))
  = \rho \left(
  \Phi\left( \frac{\delta \tL^{c}\sigma^{cu}} {\sigma^{cu}}
  \right) \right) 
\end{split} \]
Since $\Phi$ is arbitrary, this proves \cref{e:wew}. 
\end{proof}

Another way to prove \cref{e:wew} is to rerun the proof of \cref{l:ty1,l:eps} to see that, let $e^{cu}$ be the vector field $e^{cu}(y_t)$ defined on $\cV^{cu}(x)$, then
\[ \begin{split}
  \frac{\delta \tL^{c}\sigma^{cu}} {\sigma^{cu}}
  = \eps^{cu}\cL_{X^c}e^{cu}
  = \eps^{cu} \cL_{\eta F} e^{cu}\\
  = \eps^{cu} \sum_i e_1\wcw \cL_{\eta F} e_{i} \wcw e_u   \wedge F.
\end{split} \]
Since  $\cL_{\eta F} e_{i} = -e_{i}(\eta)F + \eta\cL_F e_i= -e_{i}(\eta)F$, we have
\[ \begin{split}
  \frac{\delta \tL^{c}\sigma^{cu}} {\sigma^{cu}}
  = - \eps^{cu} \sum_i e_1\wcw e_{i}(\eta) F \wcw e_u  \wedge F.
\end{split} \]
Since $F \wedge F=0$, all terms in the sum are zero except for the last, so
\[ \begin{split}
  \frac{\delta \tL^{c}\sigma^{cu}} {\sigma^{cu}}
  = - \eps^{cu}F(\eta) e^{cu}
  = - F(\eta).
\end{split} \]

\begin{proof}[Formal but direct proof of \cref{t:divuu}]
Fix an arbitrary large number $T$, define the $u$-covector $\eps$ and the $u$-vector $e$ field  in a small neighbourhood of $x$ on $\cV^u(x)$ as
\[ \begin{split}
  \eps(y) = f^{*-T} \teps(y_{-T}),
  \quad \textnormal{} \quad 
  e(y) = f_*^{T} \te(y_{-T}),
  \quad \textnormal{} \quad 
  y\in \cV^u(x),
\end{split} \]
so $\eps e=1$.
In this proof let $\sigma^u$ be the conditional density of the measure obtained by pushing forward the Lebesgue measure for time $T$, so
$$\frac{\sigma^u(x)}{\sigma^u(y)}=\frac{|e(y)|}{|e(x)|}=\frac{|\eps(x)|}{|\eps(y)|}$$
The re-distribution of the conditional measure caused by $X^{u}$ is
\begin{equation} \begin{split} \label{e:xiaoyue}
  \frac{\delta \tL^u\sigma^u}{\sigma^u}
  = \frac{(\cL_{-X^{u}} \eps) e}{\eps e}
  = \eps (\cL_{X^{u}} e) 
  = \eps \nabla_{X^{u}}e - \eps \nabla_e X^u 
  \\
  = \eps  \left(\nabla_{X^{u}}e - \nabla_e X + \nabla_e X^c +\nabla_e X^s  \right)
\end{split} \end{equation}

For the second term in \eqref{e:xiaoyue}, by definition of $\div^v$,
\[ \begin{split}
   \eps  \nabla_e X =: \div^v X.
\end{split} \]

For the third term in \eqref{e:xiaoyue}, we can formally write
\[ \begin{split}
   \eps \nabla_e X^c 
   = \eps \nabla_e(\eta F)
   = \eps \sum_{i=1}^u e_1\wcw \left(e_i(\eta) F + \eta \nabla_{e_i}(F)\right) \wcw e_u 
   = \eta \eps \nabla_{e}F
   =: \eta \div^v F,
\end{split} \]
where $\eta:= \teps^c(X).$ 
The last equality uses that $\eps_i(F)=0$ for $\eps_i\in V^{*u}$.
The last expression is well defined since $F$ is smooth.

For the first term in \eqref{e:xiaoyue}, use lemma~\ref{l:vari},
\[ \begin{split}
  \eps \nabla_{X^{u}} e
  = \eps f_*^T \nabla_{f_*^{-T}X^{u}} e_{-T}
  + \eps \int_0^T f_*^{T-t} \nabla^2_{f_*^{t-T}X^{u},  e_{t-T}}F_{t-T} dt
  \\
  = \eps_{-T} \nabla_{f_*^{-T}X^{u}} e_{-T}
  + \int_0^T \eps_{t-T} \nabla^2_{f_*^{t-T}X^{u},  e_{t-T}}F_{t-T} dt.
\end{split} \]
The first terms goes to zero as $T\rightarrow 0$.
Denote $\om:=\eps\nabla^2_{\slot,  e}F$, then
\[ \begin{split}
  \eps \nabla_{X^{u}} e
  \rightarrow \int_{-\infty}^0 \om_{t}  f_*^{t}X^{u} dt
  = X^u \int_{-\infty}^0  f^{*t} \om_{t} dt,
  \quad \textnormal{as} \quad
  T\rightarrow \infty.
\end{split} \]

\begin{figure}[ht] \centering
  \includegraphics[width=0.8\textwidth]{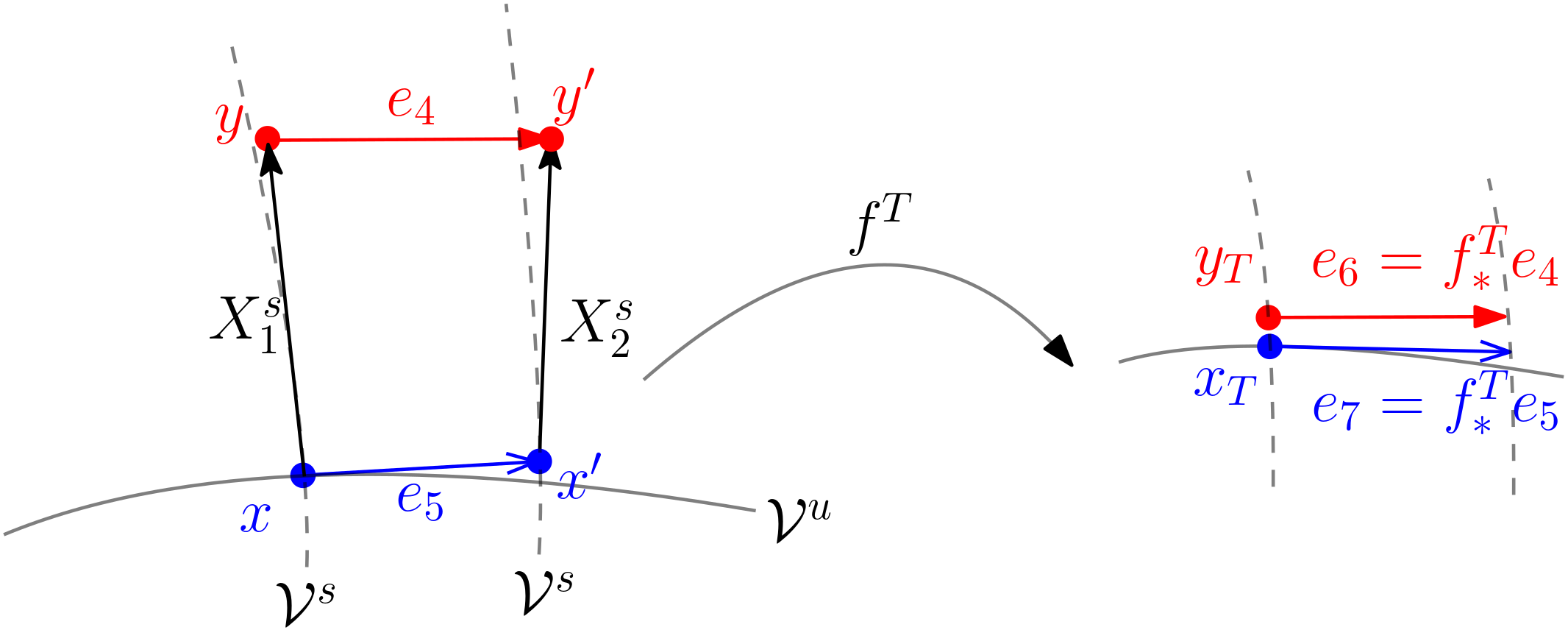}
  \caption{Intuitions for $\nabla_eX^s$. $x'-x=e_5=e(x)$, $y-x=X^s_1:=X^s(x)$,  $y'-x'=X^s_2:=X^s(x')$, and $e_4:=y'-y$. 
  Note that here $e_4$ is pushforward of $e$ along $X^s$: this is different from figure~\ref{f:alau}.}
  \label{f:bala}
\end{figure}

For the last term in~\eqref{e:xiaoyue}, the intuitive explanation for $\nabla_eX^s$  is illustrated in figure~\ref{f:bala}.
Intuitively, assuming $u=1$ and $\cM=\R^M$,
\[ \begin{split}
  \nabla_eX^s
  = X^s_2 - X^s_1 
  = (y'-x') - (y-x)
  = (y'-y) - (x'-x)
  = e_4-e_5
\end{split} \]
More specifically, if we denote the flow of $X^s$ by $\xi$, and $\xi_*e$ be a vector field on $\cV^s(x)$, which is the pushfoward vector field of $e$ by $\xi^\tau$ for a small interval of $\tau$.
Notice that here $\xi_*e$ is different from $e_{t\ge0}$ in the main body of the paper by a constant.
Then $\cL_{X^s}\xi_*e = 0$, so that $\nabla_{X^s}{e}=\nabla_{X^s}{\xi_*e|_{\tau=0}}$ is well defined and
\[ \begin{split}
  \nabla_{e} X^s 
  = \nabla_{\xi_*e|_{\tau=0}} X^s 
  = \nabla_{X^s}{\xi_*e}.
\end{split} \]
Apply lemma~\ref{l:vari}, at $\tau=0$,
\[ \begin{split}
  \eps \nabla_{X^s} \xi_*e
  = \eps_T f_*^T \nabla_{X^s} \xi_*e
  = \eps_T \nabla_{f_*^T X^s} f_*^T \xi_*e
  - \eps_T \int_0^T f_*^{T-t} \nabla^2_{f_*^t X^s,  e_t}F_t dt.
\end{split} \]
Notice that the $\xi$ can be seen as the holonomy map of local stable foliation $\cV^s$. Since the holonomy map formally tends to identity, we can formally have
\begin{equation} \begin{split} \label{e:ddd}
  \eps_T \nabla_{f_*^T X^s} f_*^T \xi_* e \rightarrow 0.
\end{split} \end{equation}
In figure~\ref{f:bala}, this means that $\eps_T(e_6-e_7)\rightarrow0$.
Although very plausible, this equation was not really proved.
The main reason is that $e$ is only a $u$-vector, so $\xi_*e$ can not be interpreted as the holonomy map in the textbook.
Holonomy map requires transverse intersections between $e$ and the stable, but the sum of their dimensions is 1 less than needed.
The detour we took in the main body of the paper is mainly because $e^{cu}$ and $\xi$ transversely intersect, so we can use a lemma of the absolute continuity offered in textbooks.

Nevertheless, if we accept the plausible formula in~\eqref{e:ddd}, then
\[ \begin{split}
  \eps \nabla_{e} X^s
  = \eps \nabla_{X^s} e
  = - \eps_T \int_0^T f_*^{T-t} \nabla^2_{f_*^t X^s,  e_t}F_t dt
  = - \int_0^T \eps_{t} \nabla^2_{f_*^t X^s, e_t}F_t dt
  = - X^s \int_0^T f^{*t}\om_t dt.
\end{split} \]
Summarizing, denote $\psi:=\div^v F$, then $\eta \div^v F=\psi \eps^c X$, and
\[ \begin{split}
  \frac{\delta \tL^u\sigma^u} {\sigma^u}
  = \eps \left(\nabla_{X^{u}}e - \nabla_e X + \nabla_e X^c +\nabla_e X^s  \right)
  \\
  = -\div^v X 
  + \psi \eps^c X
  - X^s \int_0^T f^{*t}\om_t dt
  + X^u \int_{-\infty}^0  f^{*t} \om_{t} dt.
\end{split} \]
By the same argument as \cref{e:ucsb}, we have $\om F=F(\psi)$, so we can let $T\rightarrow\infty$ and apply the adjoint shadowing lemma to get,
\[ \begin{split}
  \frac{\delta \tL^u\sigma^u} {\sigma^u}
  = -\div^v X - \cS(\div^v\nabla F,\div^vF)X.
\end{split} \]
\end{proof}

\bibliographystyle{abbrv}
{\footnotesize\bibliography{library}}

\end{document}